\documentclass[sn-mathphys-num]{sn-jnl}
\usepackage{graphicx}
\usepackage{multirow}
\usepackage{amsmath,amssymb,amsfonts}
\usepackage{amsthm}
\usepackage{mathrsfs}
\usepackage[title]{appendix}
\usepackage{xcolor}
\usepackage{textcomp}
\usepackage{manyfoot}
\usepackage{booktabs}
\usepackage{algorithm}
\usepackage{algorithmicx}
\usepackage{algpseudocode}
\usepackage{listings}
\newtheorem{theorem}{Theorem}
\newtheorem{definition}{Definition}
\newcommand{\C}{{\mathbb{C}}}

\newcommand{\N}{\mathbb{N}}

\renewcommand{\P}{{\mathbb{P}}}

\newcommand{\R}{{\mathbb{R}}}

\newcommand{\bbf}{\mathrm{\textbf{f}}}
\newcommand{\ddc}{\mathrm{{\rm dd^c}}}
\newcommand{\bbF}{\mathrm{\textbf{F}}}

\newtheorem{corollary}[equation]{Corollary}

\newtheorem{lemma}[equation]{Lemma}

\begin{document}

\title[General form of second main theorem on generalized $p$-Parabolic manifolds]{General form of second main theorem on generalized $p$-Parabolic manifolds for arbitrary closed subschemes}

\author{\fnm{Duc Quang} \sur{Si}} \email{quangsd@hnue.edu.vn}
\affil{\orgdiv{Department of Mathematics}, \orgname{Hanoi National University of Education}\\ \orgaddress{\street{136-Xuan Thuy, Cau Giay}, \city{Hanoi}, \country{Vietnam}}}


\abstract{By introducing the notion of distributive constant for a family of closed subschemes, we establish a general form of the second main theorem for algebraic non-degenerate meromorphic mappings from a generalized $p$-Parabolic manifold into a projective variety with arbitrary families of closed subschemes. As its consequence, we give a second main theorem for such meromorphic mappings intersecting arbitrary hypersurfaces with an explicitly truncation level for the counting functions. }




\maketitle

\section{Introduction}

\noindent
In 1933, H. Cartan \cite{Ca} established the second main theorem for linearly non-degenerate holomorphic curves from $\C$ into $\P^n(\C)$ intersecting fixed hyperplanes in general position. In \cite{St77,St81} W. Stoll initially investigated the second main theorem for the case of meromorphic mappings from parabolic manifolds into projective spaces with hyperplanes in general position. Recently, motivated by the work of Ru \cite{Ru09} for holomorphic curves from $\C$ into projective varieties, Q. Han \cite{Han} has established a second main theorem for algebraically no-degenerate meromorphic mappings from a generalized $p$-Parabolic manifold into a projective variety with hypersurfaces in general position. In order to state the result of Han, we recall the following notions from \cite{Han}.

Let $(M, \omega)$ be a connected K\"{a}hler manifold of dimension $m$ with the K\"{a}hler form $\omega$. Let $p$ be an integer, $1\le p\le m$. Assume that there is a plurisubharmonic function $\phi$ with
\begin{itemize}
\item[(i) ] $\{\phi=-\infty\}$ is a closed subset of $M$ with strictly lower dimension,
\item[(ii)] $\phi$ is smooth on the open dense set $M \backslash\{\phi=-\infty\}$, with $\ddc \phi \geqslant 0$, such that
$$\left(\ddc \phi\right)^{p-1} \wedge \omega^{m-p} \not \equiv 0 \quad \text { and } \quad\left(\ddc \phi\right)^{p} \wedge \omega^{m-p} \equiv 0.$$
\end{itemize}
We define
$$\tau:=\mathrm{e}^{\phi} \quad \text { and } \quad \sigma:=d^{c} \phi \wedge(\ddc\phi)^{p-1} \wedge \omega^{m-p}.$$
The triple $(M,\tau,\omega)$, or just simple $M$, is said to be a generalized $p$-Parabolic manifold. The function $\tau$ is non-negative and is called a $p$-Parabolic exhaustion on $M$. We have:
\begin{align}
\label{1.1}
&(\ddc\tau)^p\wedge\omega^{m-p}\not\equiv 0,\ \ \ \ d\sigma=(\ddc\phi)^p\wedge\omega^{m-p}\equiv 0,\\ 
\label{1.2}
& (\ddc\tau)^j=\tau^j\left( (\ddc\phi)^j+jd\phi\wedge d^c\phi\wedge (\ddc\phi)^{j-1}\right)\text{ for } j=1,2,\ldots,p. 
\end{align}
We set $\Omega:=(\ddc\tau)^p\wedge\omega^{m-p}$, which is a volume form on $M$.

Let $(M,\tau,\omega)$ be a generalized $p$-Parabolic manifold. For $r\ge 0$ and $A\subset M$, define
\begin{align*}
M[r]&=\{x\in A|\ \tau (x)\le r^2\}, &M(r)=\{x\in M|\ \tau (x)< r^2\},\\ 
M\langle r\rangle &=\{x\in M|\ \tau (x)= r^2\}, &M_*=\{x\in M|\ \tau(x)>0\}.
\end{align*}
From Stokes's formula and (\ref{1.1}), for any $r>0$ it follows immediately that
\begin{align}\label{1.3}
\int_{M\langle r\rangle} \sigma=\kappa,
\end{align}
where $\kappa$ is a constant dependent only upon the structure of $M$. Also, by the Green-Jensen formula on generalized $p$-Parabolic manifolds (see \cite[Theorem 1.3]{WW}), if $\gamma$ is a plurisubharmonic function on $M$, then for every $r>s>0$, we have
$$\int_{s}^{r} \frac{\mathrm{d} t}{t^{2 p-1}} \int_{M[t]} \ddc[\gamma] \wedge(\ddc \tau)^{p-1} \wedge \omega^{m-p}=\frac{1}{2} \int_{M\langle r\rangle} \gamma \sigma-\frac{1}{2} \int_{M\langle s\rangle}\gamma\sigma,$$
where the operation $\ddc[\gamma]$ is taken in the sense of currents.

Throughout this paper, we fix a homogeneous coordinate system $(x_0:\cdots :x_N)$ of $\P^N(\C)$. Let $f$ be a meromorphic mapping from $M$ into $\P^N(\C)$. For each $x\in M$, there exist an open neighborhood $U$ of $x$ and a locally reduced representation $\bbf =(\bbf_0,\ldots ,\bbf_N): U\rightarrow\C^{N+1}$ of $f$ on $U$. If there is another local representation $\tilde\bbf$ of $f$ on a subdomain $V$ of $M$, then there is a nowhere zero holomorphic function $g_{U,V}$ such that
$$\bbf =g_{U,V}\tilde\bbf\ \text{ on }U\cap V.$$
Let $B$ be a holomorphic form of bi-degree $(m-1,0)$ on $M$. Take a chart $z=(z_1,\ldots ,z_m)$ with $U_z\cap U\ne\emptyset$. The $B$-derivative $\bbf'=\bbf'_{B,z}$ for $z$ is defined on $U_z\cap U$ by
$$ d\bbf\wedge B=\bbf' dz_1\wedge\cdots\wedge dz_m. $$
The $k$-th $B$-derivative $\bbf^{(k)}$ of $\bbf$ is defined by: $d\bbf^{(k-1)}\wedge B=\bbf^{(k)}dz_1\wedge\cdots\wedge d_{z_m}$. Put $\bbf^{(0)}=\bbf$ and consider
$$ \bbf_k=\bbf\wedge\bbf'\wedge\cdots\wedge\bbf^{(k)}: U\rightarrow\wedge^{k+1}\C^{N+1}. $$
The generated $k$-th associated map $f_k$ of $f$ is defined by 
$$f_{k}:=[\bbf_{k}]: M \rightarrow \P\left(\wedge^{k+1} \mathbb{C}^{n+1}\right)=\P^{n_{k}}(\C) \text{ for }n_{k}=\binom{N+1}{k+1}.$$
The following general conditions are proposed by Han in \cite{Han}.
\begin{itemize}
\item[($\mathcal{A}_{1}$)] $(M, \tau, \omega)$ denotes a generalized $p$-Parabolic manifold which possesses a globally defined meromorphic $(m-1,0)$-form $B$ such that, for any linearly non-degenerate meromorphic map $f: M \rightarrow \P^N(\C)$, the $k$-th associated map $f_{k}$ is well defined for $k=0,1, \ldots, n$, where we set $f_0:=f$ and where $f_N$ is constant.
\item[($\mathcal{A}_{2}$)] There exists a Hermitian holomorphic line bundle $(\mathfrak{L},\mathfrak{h})$ that admits a holomorphic section $\mu$ such that, for some increasing function $\mathcal Y(\tau)$, we have
\begin{align}\label{1.4}
m i_{m-1}|\mu|_{\mathfrak{h}}^{2} B \wedge\bar B \le\mathcal Y(\tau)(\ddc \tau)^{p-1}\wedge \omega^{m-p},
\end{align}
where we write the index
$$i_{m-1}=\left(\frac{i}{2\pi}\right)^{m-1}(m-1)!(-1)^{(m-1)(m-2)/2}.$$
\end{itemize}

In this paper, we do not distinguish each hypersurface of $\P^N(\C)$ with its defining homogeneous polynomial. Let $Q$ be a hypersurface of degree $d$ in $\P^N(\C)$ defined by
$$ Q({\bf x})=\sum_{I \in \mathcal{T}_{d}}a_{I}{\bf x}^{I},$$
where $\mathcal{T}_{d}:=\{(i_0,\ldots, i_N)\in\N_0^{N+1}: i_0+\cdots+i_N=d\}$, ${\bf x}^{I}=x_0^{i_0}\ldots x_N^{i_N}$ for $I=(i_0,\ldots,i_N)\in\mathcal T_d$,  $a_I\ (I\in\mathcal T_d)$ are constants, not all zero. 
Let $f: M \rightarrow \P^N(\C)$ be a meromorphic mapping such that $f(M)\not\subset Q$. The Weil function of $f$ with respect to $Q$ is defined (locally) by
$$\lambda_{Q}(f)=\log\frac{\|\bbf\|^{d}\cdot\|Q\|}{|Q(\bbf)|},$$
where $\bbf=(\bbf_0,\ldots,\bbf_N)$ is a reduced representation $f$ on a local holomorphic coordinate chart $(z,U_{z})$, $\|\bbf\|=(\sum_{i=0}^N\|\bbf_i\|^{2})^{1/2}$ and $\|Q\|=(\sum_{I}\|a_I\|^{2})^{1/2}$. We note that the above definition does not depend on the choice of the reduced representations and then is well-defined. The proximity function and the truncated (to level $k\ge 1$) counting function of $f$ with respect to $Q$ are defined, respectively, by:
\begin{align*}
m_{f}(r,Q)&:=\int_{M\langle r\rangle}\lambda_{Q}(f)\sigma,\\ 
N^{[k]}_{f}(r,s;Q)&:=\int_{s}^{r}\frac{dt}{t^{2p-1}}\int_{M[t]}[\theta_{f}^{k,Q}]\wedge\left(\ddc\tau\right)^{p-1}\wedge\omega^{m-p}, 
\end{align*}
where $[\theta_{f}^{k,Q}]$ denotes the current generated by the truncated divisor $\theta_{f}^{k,Q}$, which is defined locally by $\theta_{f}^{k,Q}(z):=\min\{k,{\rm div}_{Q(\bbf)}(z)\}$ for a local reduced representation $\bbf$ of $f$ on a local chart $(z,U_z)$. If $k=+\infty$, we remove the character $k$ from $\theta_{f}^{k,Q}$ and the character $[k]$ from $N^{[k]}_{f}(r,s;Q)$.

Let $\Omega_{FS}$ be the Fubini-Study metric on $\P^N(\C)$. The characteristic function of $f$, for a fixed $s>0$ and any $r>s$, defined by
$$T_{f}(r,s)=\int_{s}^{r}\frac{dt}{t^{2p-1}}\int_{M[t]} f^*\Omega_{FS}\wedge(\ddc\tau)^{p-1}\wedge\omega^{m-p}.$$
The First Main Theorem is stated as follows:
$$dT_{f}(r,s)=N_{f}(r,s;Q)+m_{f}(r,Q)-m_{f}(s,Q).$$

For each $1\le k \le N-1$, we define (locally) the following auxiliary function
$$\Psi_{k}=\frac{mi_{m-1}f_{k}^*\Omega_{FS}^{k}\wedge B\wedge\bar{B}}{(\ddc\tau)^{p}\wedge \omega^{m-p}}=\frac{\|\bbf_{k-1}\|^2\cdot\|\bbf_{k+1}\|^{2}}{\|\bbf_{k}\|^{4}}\frac{1}{A_{p}},$$
where $\Omega_{FS}^{k}$ is the Fubini-Study metric on $\P\left(\wedge^{k+1} \C^{N+1}\right)$, and $A_{p},1\le p\le m$, is the $p$-th symmetric polynomial of the matrix $(\tau_{a \bar{b}})$ with respect to the K\"{a}hler metric $\omega$. 
We set
\begin{align}\label{1.5}
{\rm Ric}_{p}(r,s):=\int_{s}^{r} \frac{d t}{t^{2 p-1}} \int_{M[t]} \theta_{A_{p}}^0 \wedge\left(\ddc \tau\right)^{p-1} \wedge \omega^{m-p},
\end{align}
where $\theta_{A_{p}}^0$ is the divisor of zeros of the holomorphic function $A_{p}$, and
\begin{align}\label{1.6}
m(\mathfrak{L};r,s):=\frac{1}{2} \int_{M\langle r\rangle} \log \frac{1}{|\mu|_{\mathfrak{h}}^{2}}\sigma-\frac{1}{2} \int_{M\langle s\rangle} \log \frac{1}{|\mu|_{\mathfrak{h}}^{2}}\sigma.
\end{align}

With the above notation, the second main theorem of Han in \cite{Han} is stated as follows.

\noindent
{\bf Theorem A.}\ {\it Let $f: M \rightarrow V \subset\P^N(\C)$ be an algebraically non-degenerate meromorphic mapping from a generalized $p$-Parabolic manifold $M$ that satisfies the general conditions $(\mathcal{A}_{1})$ and $(\mathcal{A}_{2})$ into a smooth projective algebraic variety $V$ of dimension $n\ge 1$, and let $Q_{1}, Q_{2}, \ldots, Q_{q}$ be $q\ (>n)$ hypersurfaces of $\P^N(\C)$ in general position with respect to $V$. Then, for any $\epsilon>0$ and $r>s>0$, we have
\begin{align*}
\|\ \sum_{j=1}^{q}\frac{1}{\deg Q_j}m_{f}(r;Q_j) \le &(n+1+\epsilon)T_{f}(r, s)\\
&+c\left(m_0(\mathcal{L};r,s)+{\rm Ric}_{p}(r,s)+\kappa \log^{+}\mathcal Y\left(r^{2}\right)+\kappa \log^+r\right),
\end{align*}
where $c\ge 1$ is a positive constant.}

\noindent
Here the notation ``$\| P$'' means the assertion $P$ hold for all $r\in (0;+\infty)$ outside a set of finite Lebesgue measure.

Our aim in this paper is to generalize Theorem A to the case where ``families of hypersurfaces in general position'' is replaced by ``families of arbitrary closed subschemes''. In order to state our result, we first recall the notion of the Weil function of a closed subscheme on a  projective variety due to K. Yamanoi \cite{Y}. 

Let $V$ be a projective variety defined over $\C$. Then the closed subschemes of $V$ are in one-to-one correspondence with their ideal sheaves of $\mathcal{O}_{V}$. 
A Weil function $\lambda_{Z}$ for a closed subscheme $Z\subset V$ is a continuous function $\lambda_{Z}: V\setminus {\rm Supp}Z \rightarrow\R$ satisfying the condition that: there exist an affine Zariski open covering $\{U_{\alpha}\}_{\alpha\in\Lambda}$ of $V$ and a system of generators $f_1^{\alpha},\ldots, f_{r_{\alpha}}^{\alpha}\in\Gamma (U_{\alpha}, \mathcal{O}_{U_{\alpha}})$ of the ideal $\Gamma(U_{\alpha},\mathcal I_Z)\subset$ $\Gamma\left(U_{\alpha}, \mathcal{O}_{U_{\alpha}}\right)$ such that
$$\lambda_{Z}(x)\equiv -\log \max _{1\le i\le r_{\alpha}}\left|f_{i}^{\alpha}(x)\right|+c(x),$$
where $c(x)$ is a continuous function on $U_\alpha$. We list here some fundamental properties of the Weil function:

(1) Let $D \subset V$ be an effective Cartier divisor on $V$. We consider $D$ as a closed subscheme and define $\lambda_{D}$ as above. More generally, if $D$ is a $\mathbb{Q}$-Cartier divisor and $nD$ is a Cartier divisor, then we define $\lambda_{D}$ to be $\frac{1}{n}\lambda_{nD}$.

(2) Let $D \subset V$ be an effective divisor on $V$ and take a Hermitian metric $\|\cdot\|$ for the corresponding line bundle $\mathcal{O}(D)$. Let $\sigma$ be a canonical section of $\mathcal{O}(D)$ uniquely determined up to a constant multiple so that ${\rm div}\sigma=D$. Then the function $\lambda_{D}(P)=-\log \|\sigma(P)\|$ is a Weil function for $D$. In particular, if $V\subset\P^N(\C)$, $Q$ is a hypersurface of $\P^N(\C)$ and $f:M\rightarrow V$ is a meromorphic mapping then the composition function $\lambda_{Q\cap V}(f)$ exactly coincides with $\lambda_Q(f)$ in the above.

(3) Let $Y$ be a closed subscheme of $V$. Then, there exist effective Cartier divisors $D_{1}, \ldots, D_{r}$ such that $Y=\bigcap_{i=1}^{r} D_{i}.$ The (local) Weil function for $Y$ is defined by
$$\lambda_{Y}=\min _{1\le i\le r}\left\{\lambda_{D_{i}}\right\}.$$

We have the definition of the Seshadri constant of a closed subscheme with respect to a nef divisor as follows.

\noindent
{\bf Definition B.} {\it Let $Y$ be a closed subscheme of a projective variety $X$ and let $\pi: \tilde{X} \rightarrow X$ be the blowing-up of $X$ along $Y$. Let $A$ be a nef Cartier divisor on $X$. The Seshadri constant $\epsilon_{Y}(A)$ of $Y$ with respect to $A$ is defined by
$$\epsilon_{Y}(A)=\sup \left\{\gamma \in \mathbb{Q}^{\geq 0} \mid \pi^* A-\gamma E \text { is } \mathbb{Q} \text {-nef }\right\},$$
where $E$ is an effective Cartier divisor on $\tilde{X}$ whose associated invertible sheaf is the dual of $\pi^{-1} \mathcal I_{Y} \cdot \mathcal{O}_{\hat{X}}$.}

We give the following definition of the notion of distributive constant of a family of closed subschemes with respect to a projective variety, which is a generalization of the notion of distributive constant of a family of hypersurfaces introduced in \cite{Qpcf}.
\begin{definition}\label{1.7}
The distributive constant of a families of closed subschemes $\mathcal Y=\{Y_{1}, \ldots, Y_{q}\}$ on a projective variety $V$ with respect to a subvariety $X\subset V$ is defined by
$$ \delta=\max_{J\subset\{1,\ldots,q\}}\max\left\{1;\frac{\sharp J}{\dim X-\dim X\cap\bigcap_{j\in J}Y_j}\right\},$$
where $\dim\emptyset=-\infty$.
\end{definition}

Our main result is stated as follows:

\begin{theorem}\label{1.8}
Let $V$ be a complex projective variety of dimension $n \geq 1$ defined over $\C$. Let $f$ be an algebraically nondegenerate meromorphic mapping from a generalized $p$-Parabolic manifold $(M,\tau,\omega)$, which satisfies the conditions $(\mathcal{A}_{1})$ and $(\mathcal{A}_{2})$, into $V$. Let $Y_{1}, \ldots, Y_{q}$ be $q$ closed subschemes in $V$ and $\delta\ge 1$ be a positive number. Suppose that there exists an ample Cartier divisor $A$ on $V$. Then, for any $\epsilon>0$ and $r>s>0$, we have
\begin{align*}
\| \int_{M\langle r\rangle}&\max _{K \in \mathcal{K}} \sum_{j \in K}\epsilon_{Y_j}(A) \lambda_{Y_{j}}(f) \sigma \le(\delta(n+1)+\epsilon) T_{f,A}(r,s)\\
&+c\left(m_0(\mathfrak{L};r,s)+{\rm Ric}_{p}(r,s)+\kappa \log ^+ \mathcal{Y}\left(r^{2}\right)+k \log ^+ r\right),
\end{align*}
where $\mathcal{K}$ is the set of all subsets $K \subset\{1, \ldots, q\}$ so that the distributive constant of $\{Y_j;j\in K\}$ with respect to $V$ does not exceed $\delta$ and $c \ge 1$ is a positive constant.
\end{theorem}

By usual arguments, from Theorem \ref{1.8}, we obtain the following second main theorem.

\begin{corollary}\label{1.9}
Let $V$ be a complex projective variety of dimension $n \geq 1$ defined over $\C$. Let $f$ be an algebraically nondegenerate meromorphic mapping from a generalized $p$-Parabolic manifold $(M,\tau,\omega)$, which satisfies the conditions $(\mathcal{A}_{1})$ and $(\mathcal{A}_{2})$, into $V$. Let $\{Y_{1}, \ldots, Y_{q}\}$ be a family of $q$ closed subschemes in $V$ with the distributive constant with respect to $V$ equal to $\delta$. Suppose that there exists an ample Cartier divisor $A$ on $V$. Then, for any $\epsilon>0$ and $r>s>0$, we have
\begin{align*}
\| \sum_{i=1}^q\epsilon_{Y_j}(A)&m_f(r,Y_j)\le(\delta(n+1)+\epsilon) T_{f,A}(r,s)\\
&+c\left(m_0(\mathfrak{L};r,s)+{\rm Ric}_{p}(r,s)+\kappa \log ^+ \mathcal{Y}\left(r^{2}\right)+k \log ^+ r\right)
\end{align*}
for some positive constant $c\ge 1$.
\end{corollary}
By simple computation, we see that if $Y_{1},\ldots,Y_{q}$ are in $l$-subgeneral position with respect to $V$ then the distributive $\delta$ is bounded above by $(l-n+1)$, in particular $\delta=1$ if they are in general position with respect to $V$. Therefore, the above corollary is a generalization of Theorem A of Q. Han and the general form of second main theorem of W. Chen- N. V. Thin \cite{CT} for the case of meromorphic mappings on $p$-Parabolic manifolds intersecting family of hypersurfaces in subgeneral position. Our above corollary also is a generalization of a recent second main theorem of L. Wang, T. Cao and H. Cao \cite{WCC} for the case of families of closed subschemes in subgeneral position.

\section{Some auxiliary results}

Let $X$ be a projective subvariety of $\P^N(\C)$ of dimension $n$ and degree $\Delta_X$. The Chow form of $X$ is the unique polynomial, up to a constant scalar, 
$$F_X(\textbf{u}_0,\ldots,\textbf{u}_n) = F_X(u_{00},\ldots,u_{0N};\ldots; u_{n0},\ldots,u_{nN})$$
in $N+1$ blocks of variables $\textbf{u}_i=(u_{i0},\ldots,u_{iN}), i = 0,\ldots,n$ with the following
properties: 
\begin{itemize}
\item $F_X$ is irreducible in $\C[u_{00},\ldots,u_{nN}]$;
\item $F_X$ is homogeneous of degree $\delta$ in each block $\textbf{u}_i, i=0,\ldots,n$;
\item $F_X(\textbf{u}_0,\ldots,\textbf{u}_n) = 0$ if and only if $X\cap H_{\textbf{u}_0}\cap\cdots\cap H_{\textbf{u}_n}\ne\emptyset$, where $H_{\textbf{u}_i}, i = 0,\ldots,n$, are the hyperplanes given by $u_{i0}x_0+\cdots+ u_{iN}x_N=0.$
\end{itemize}
Let ${\bf c}=(c_0,\ldots, c_N)$ be a tuple of real numbers and $t$ be an auxiliary variable. We write
\begin{align*}
F_X(t^{c_0}u_{00},&\ldots,t^{c_N}u_{0N};\ldots ; t^{c_0}u_{n0},\ldots,t^{c_N}u_{nN})\\ 
& = t^{e_0}G_0(\textbf{u}_0,\ldots,\textbf{u}_N)+\cdots +t^{e_r}G_r(\textbf{u}_0,\ldots, \textbf{u}_N),
\end{align*}
with $G_0,\ldots,G_r\in K[u_{00},\ldots,u_{0N};\ldots; u_{n0},\ldots,u_{nN}]$ and $e_0>e_1>\cdots>e_r$. The Chow weight of $X$ with respect to ${\bf c}$ is defined by
\begin{align*}
e_X({\bf c}):=e_0.
\end{align*}
For $\textbf{u} = (u_0,\ldots,u_R),\textbf{v} = (v_0,\ldots,v_R)\in\mathbb R^{R+1}$, we define
$$ \textbf{u}\cdot \textbf{v}:=\sum_{i=1}^Ru_iv_i. $$
For $\textbf{a} = (a_0,\ldots,a_N)\in\mathbb Z^{N+1}$ we write ${\bf x}^{\bf a}$ for the monomial $x^{a_0}_0\cdots x^{a_N}_N$. Let $I=I_X$ be the prime ideal in $\C[x_0,\ldots,x_N]$ defining $X$. Let $\C[x_0,\ldots,x_N]_m$ denote the vector space of homogeneous polynomials in $\C[x_0,\ldots,x_N]$ of degree $m$ (including $0$). Put $I_m :=\ C[x_0,\ldots,x_N]_m\cap I$ and define the Hilbert function $H_X$ of $X$ by, for $m = 1, 2,...,$
$$H_X(m):=\dim (\C[x_0,...,x_N]_m/I_m).$$
The $m$-th Hilbert weight $S_X(m,{\bf c})$ of $X$ with respect to the tuple ${\bf c}=(c_0,\ldots,c_N)\in\mathbb R^{N+1}$ is defined by
$$S_X(m,{\bf c}):=\max\biggl (\sum_{i=1}^{H_X(m)}{\bf a}_i\cdot{\bf c}\biggl),$$
where the maximum is taken over all sets of monomials ${\bf x}^{{\bf a}_1},\ldots,{\bf x}^{{\bf a}_{H_X(m)}}$ whose residue classes modulo $I$ form a basis of $\C[x_0,\ldots,x_N]_m/I_m.$
\begin{theorem}[{\cite[Theorem 4.1]{EverFer2}, see also \cite[Theorem 2.1]{Ru09}}]\label{2.1}
Let $X\subset\P^N(\C)$ be an algebraic variety of dimension $n$ and degree $D$. Let $u>D$ be an integer and let ${\bf c}=(c_0,\ldots,c_N)\in\mathbb R^{N+1}_{\geqslant 0}$.
Then
$$ \frac{1}{uH_X(u)}S_X(u,{\bf c})\ge\frac{1}{(n+1)D}e_X({\bf c})-\frac{(2n+1)D}{u}\cdot\left (\max_{i=0,...,N}c_i\right). $$
\end{theorem}

\begin{theorem}[{see \cite[Theorem 3.2]{Qjnt}}]\label{2.2}
Let $V$ be a projective subvariety of $\P^N(\C)$ of dimension $n\ge 1$ and degree $\Delta_V$ and $(x_0:\cdots:x_N)$ be homogeneous coordinates system of $\P^N(\C)$. Let ${\bf c}=(c_0,\ldots,c_N)$ be a tuple of non-negative reals. Let $H_0,\ldots,H_N$ be hyperplanes in $\P^N(\C)$ defined by $H_{i}=\{y_{i}=0\}\ (0\le i\le N)$. Let $\{i_0,\ldots, i_m\}$ be a subset of $\{0,\ldots,N\}\ (m\ge n)$ such that:
\begin{itemize}
\item[(1)] $c_{i_m}=\min\{c_{i_0},\ldots,c_{i_m}\}$,
\item[(2)] $V\cap\bigcap_{j=0}^{m-1}H_{i_j}\ne\emptyset$, 
\item[(3)] and $V\not\subset H_{i_j}$ for all $j=0,\ldots,m$.
\end{itemize}
Let $\delta_V$ be the distributive constant of the family $\{H_{i_j}\}_{j=0}^m$ with respect to $V$. Then
$$e_V({\bf c})\ge \frac{\Delta_V}{\delta_V}(c_{i_0}+\cdots+c_{i_m}).$$
\end{theorem}

\begin{theorem}[{see \cite[Theorem 4.1]{Han}}]\label{2.3}
 Let $f: M \rightarrow \P^N(\C)$ be a linearly non-degenerate meromorphic map defined on a $p$-Parabolic manifold $M$ satisfying the general conditions $(\mathcal{A}_{1})$ and $(\mathcal{A}_{2})$, and let $H_1, \ldots, H_q$ be $q$ hyperplanes. Then, for $r>s>0$, we have
\begin{align*}
\bigl\| \int_{M\langle r\rangle}&\max _{K} \sum_{j \in K} \lambda_{H_{j}}(f) \sigma \le (N+1) T_{f}(r,s)-N_{ramf}(r,s) \\
&+\frac{N(N+1)}{2}\left(m_{o}(\mathfrak{L};r,s)+{\rm Ric}_{p}(r,s)+\kappa \log ^+ T_{f}(r,s)\right) \\
&+\frac{\kappa N(N+1)}{2}\left(\log ^+ m_{o}(\mathfrak{L};r,s)+\log ^+ \mathcal{Y}\left(r^{2}\right)+\log ^+ {\rm Ric}_{p}(r,s)+\log ^+ r\right),
\end{align*}
where maximum is taken over all subsets $K$ of $\{1, \ldots, q\}$ such that the generating linear forms of the hyperplanes in each set are linearly independent and $N_{ramf}(r,s)$ is the counting function of the ramification divisor ${\rm div}\bbf_N$.
\end{theorem}

\begin{lemma}[{see \cite[Lemma 2.5.2]{V87} and also \cite[Theorem 2.1(h)]{Sil87}}]\label{2.4} 
Let $Y$ be a closed subscheme of $V$, and let $\tilde{V}$ be the blowing-up of $V$ along $Y$ with exceptional divisor $E$. Then $\lambda_{Y}(\pi(P))=\lambda_{E}(P)+O(1)$ for $P \in \tilde{V} \setminus{\rm Supp} E$. 
\end{lemma}

\begin{lemma}[{see \cite[Lemma 5.4 .24]{Laz17}}]\label{2.5}
Let $V$ be a projective variety, $\mathcal{I}$ be a coherent ideal sheaf. Let $\pi: \tilde{V} \rightarrow V$ be the blowing-up of $\mathcal{I}$ with exceptional divisor $E$. Then there exists an integer $p_0=p_0(\mathcal{I})$ with the property that if $p \geq p_0$, then $\pi_{*} \mathcal{O}_{\tilde{V}}(-p E)=\mathcal{I}^{p}$, and moreover, for any divisor $D$ on $V$,
$$H^{i}\left(V, \mathcal{I}^{p}(D)\right)=H^{i}\left(\tilde{V}, \mathcal{O}_{\tilde{V}}\left(\pi^* D-p E\right)\right)$$
for all $i \geq 0$.
\end{lemma}
%

\section{Proof of main theorem}
In order to prove the main theorem with closed subschemes, we first prove the general form of second main theorem for arbitrary families of hypersurfaces.
\begin{theorem}\label{3.1}
Let $V \subset \P^N(\C)$ be a smooth complex projective variety of dimension $n \geq 1$, $(M,\tau,\omega)$ be a generalized $p$-Parabolic manifold which satisfies the conditions $(\mathcal{A}_{1})$ and $(\mathcal{A}_{2})$. Let $f: M \rightarrow V$ be an algebraically non-degenerate meromorphic mapping. Let $Q_{1}, \ldots, Q_{q}$ be $q$ hypersurfaces in $\P^N(\C)$. Let $\delta\ge 1$ be a positive integer. Then, for any $\epsilon>0$ and $r>s>0$, we have
\begin{align*}
\big\| \int_{M\langle r\rangle}\max _{K \subset \mathcal{K}}&\sum_{j \in K}\frac{1}{\deg Q_j} \lambda_{Q_{j}}(f) \sigma\le(\delta(n+1)+\epsilon)T_{f}(r,s) \\
&+c\left(m_0(\mathfrak{L}; r, s)+{\rm Ric}_{p}(r,s)+\kappa \log^+\mathcal{Y}\left(r^{2}\right)+k \log^+r\right),
\end{align*}
where $\mathcal{K}$ is the set of all subsets $K \subset\{1, \ldots, q\}$, such that the family of hypersurfaces $\left\{Q_{j}, j \in K\right\}$ has distributive constant with respect to $V$ not exceed $\delta$.
\end{theorem}
\begin{proof}
By adding more hypersurfaces into each set $\{Q_j;j\in K\}$ if necessary, without lose of generality we may suppose that $V\cap\bigcap_{j=1}^qQ_j=\emptyset$ and $\mathcal K$ be the set of all subset $K\subset\{1,\ldots,q\}$ such that $V\cap\bigcap_{j\in K}Q_j=\emptyset$ and the distributive constant of the family $\{Q_j\}_{j\in K}$ with respect to $V$ does not exceed $\delta$. Also, by replacing $Q_j$ with $Q_j^{d/\deg Q_j}$ if necessary, where $d={\rm lcm}(\deg Q_1,\ldots,\deg Q_q)$, we may assume that $\deg Q_j=d$ for every $j=1,\ldots,q$.

Consider the mapping $\Phi$ from $V$ into $\P^{q-1}(\C)$, which maps a point $x\in V$ into the point $\Phi(x)\in\P^{q-1}(\C)$ given by
$$\Phi(x)=(Q_1(x):\cdots:Q_{q}(x)).$$ 
Let $Y=\Phi (V)$. Since $V\cap\left (\bigcap_{j=1}^{q}Q_{j}\right )=\emptyset$, $\Phi$ is a finite morphism on $V$ and $Y$ is a complex projective subvariety of $\P^{q-1}(\C)$ with $\dim Y=n$ and $D:=\deg Y\le d^k.\deg V$. 
For every 
$${\bf a} = (a_{1},\ldots ,a_{q})\in\mathbb Z^q_{\ge 0}$$ 
and
$${\bf y} = (y_{1},\ldots ,y_{q})$$ 
we set ${\bf y}^{\bf a}:= y_{1}^{a_{1}}\ldots y_{q}^{a_{q}}$. Let $u$ be a positive integer and define:
\begin{align*}
n_u:=&H_Y(u)-1,\ l_u:=\binom{q+u-1}{u}-1;\\
Y_u:=&\C[y_1,\ldots,y_q]_u/(I_Y)_u.
\end{align*}
Then $Y_u$ is a vector space of dimension $n_u+1$. Fix a basis $\{v_0,\ldots, v_{n_u}\}$ of $Y_u$ and consider the meromorphic mapping $F$ with local reduced representations
$$\bbF=(v_0(\Phi\circ \bbf),\ldots ,v_{n_u}(\Phi\circ \bbf)):\C^p\rightarrow \C^{n_u+1},$$
where the maps $\bbf$ are local reduced representation of $f$. Hence $F$ is linearly non-degenerate, since $f$ is algebraically non-degenerate.

Fix a point $z\not\in\bigcup_{i=1}^q(Q_j(\bbf))^{-1}(0)$ for some local representation $\bbf$ of $f$ on an open neighborhood $U$ of $z$ in $M$. We define 
$${\bf c}_z:= (c_{1,z},\ldots,c_{q,z})\in\mathbb R^{q},$$ 
where
$$c_{j,z}:=\log\frac{\|\bbf(z)\|^d\|Q_{j}\|}{|Q_{j}(\bbf)(z)|}\ge 0\text{ for } j=1,...,q.$$
By the definition of the Hilbert weight, there are ${\bf a}_{1,z},...,{\bf a}_{H_Y(u),z}\in\mathbb Z^{q}_{\ge 0}$ with
$$ {\bf a}_{i,z}=(a_{i,1,z},\ldots, a_{i,q,z})\text{ with }a_{i,s,z}\in\{1,...,l_u\}, $$
 such that the residue classes modulo $(I_Y)_u$ of ${\bf y}^{{\bf a}_{1,z}},...,{\bf y}^{{\bf a}_{H_Y(u),z}}$ form a basis of the vector space $\C[y_1,...,y_q]_u/(I_Y)_u$ and
$$S_Y(u,{\bf c}_z)=\sum_{i=1}^{H_Y(u)}{\bf a}_{i,z}\cdot{\bf c}_z.$$
Since ${\bf y}^{{\bf a}_{i,z}}\in Y_u$ (modulo $(I_Y)_u$), we may write
$$ {\bf y}^{{\bf a}_{i,z}}=L_{i,z}(v_0,\ldots ,v_{H_Y(u)}), $$ 
where $L_{1,z},\ldots, L_{ H_Y(u),z}$ are linearly independent linear forms with coefficients in $\C$. Thus
\begin{align*}
\log\prod_{i=1}^{H_Y(u)} |L_{i,z}(\bbF(z))|&=\log\prod_{i=1}^{H_Y(u)}\prod_{1\le j\le q}|Q_{j}(\bbf(z))|^{a_{i,j,z}}\\
&=-S_Y(m,{\bf c}_z)+duH_Y(u)\log \|\bbf(z)\| +O(uH_Y(u)).
\end{align*}
It follows that
\begin{align*}
\log\prod_{i=1}^{H_Y(u)}\dfrac{\|\bbF(z)\|\cdot \|L_{i,z}\|}{|L_{i,z}(\bbF(z))|}=&S_Y(u,{\bf c}_z)-duH_Y(u)\log \|\bbf(z)\| \\
&+H_Y(u)\log \|\bbF(z)\|+O(uH_Y(u)).
\end{align*}
Note that $L_{i,z}$ depends on $i$, $z$ and $u$, but the number of these linear forms is finite. Denote by $\mathcal L$ the set of all $L_{i,z}$ occurring in the above inequalities. We have
\begin{align}\label{3.2}
\begin{split}
S_Y(u,{\bf c}_z)\le&\max_{\mathcal J\subset\mathcal L}\log\prod_{L\in \mathcal J}\dfrac{\|\bbF(z)\|\cdot \|L\|}{|L(\bbF(z))|}+duH_Y(u)\log \|\bbf(z)\|\\
& -H_Y(u)\log \|\bbF(z)\|+O(uH_Y(u)),
\end{split}
\end{align}
where the maximum is taken over all subsets $\mathcal J\subset\mathcal L$ with $\sharp\mathcal J=H_Y(u)$ and $\{L;L\in\mathcal J\}$ is linearly independent.
From Theorem \ref{2.1} we have
\begin{align}\label{3.3}
\dfrac{1}{uH_Y(u)}S_Y(u,{\bf c}_z)\ge&\frac{1}{(n+1)D}e_Y({\bf c}_z)-\frac{(2n+1)D}{u}\max_{1\le j\le q}c_{j,z}.
\end{align}
It is clear that
\begin{align*}
\max_{1\le j\le q}c_{j,z}\le \sum_{1\le j\le q}\log\frac{\|\bbf(z)\|^d\|Q_{j}\|}{|Q_{j}(\bbf)(z)|}+O(1),
\end{align*}
where $O(1)$ does not depend on $z$.
From (\ref{3.2}), (\ref{3.3}) and the above inequality, we have
\begin{align}
\label{3.4}
\begin{split}
\frac{1}{(n+1)D}e_Y({\bf c}_z)\le&\dfrac{1}{uH_Y(u)}\left (\max_{\mathcal J\subset\mathcal L}\log\prod_{L\in \mathcal J}\dfrac{\|\bbF(z)\|\cdot \|L\|}{|L(\bbF(z))|}-H_Y(u)\log \|\bbF(z)\|\right )\\
&+d\log \|\bbf(z)\|+\frac{(2n+1)D}{u}\max_{1\le j\le q}c_{j,z}+O(1)\\
\le &\dfrac{1}{uH_Y(u)}\left (\max_{\mathcal J\subset\mathcal L}\log\prod_{L\in\mathcal J}\dfrac{\|\bbF(z)\|\cdot \|L\|}{|L(\bbF(z))|}-H_Y(u)\log \|\bbF(z)\|\right )\\
&+d\log \|\bbf(z)\|+\frac{(2n+1)D}{u}\sum_{1\le j\le q}\log\frac{\|\bbf(z)\|^d\|Q_{j}\|}{|Q_{j}(\bbf)(z)|}+O(1).
\end{split}
\end{align}
Consider an element $K=\{j_0,\ldots,j_l\}\in\mathcal K$ with $1\le j_0<j_1<\cdots <j_l\le q$ and suppose
$$c_{j_0,z}\ge \cdots \ge c_{j_l,z}. $$
Let $m$ be the smallest integer such that $V\cap\bigcap_{i=0}^mQ_{j_i}=\emptyset$. Then by Theorem \ref{2.2}, we have
$$e_Y({\bf c}_z)\ge \frac{D}{\delta}\cdot(c_{j_0,z}+\cdots +c_{j_m,z})\ge \frac{D}{\delta}\cdot(c_{j_0,z}+\cdots +c_{j_l,z})+O(1),$$
where $O(1)$ does not depend on $z$. This implies that
\begin{align}\label{3.5}
e_Y({\bf c}_z)\ge \frac{D}{\delta}\cdot\max_{K\in\mathcal K}\log\prod_{j\in K}\frac{\|\bbf(z)\|^d\|Q_j\|}{|Q_j(\bbf)(z)|}+O(1).
\end{align}
Then, from (\ref{3.4}) and (\ref{3.5}), we have
\begin{align}\label{3.6}
\begin{split}
\frac{1}{\delta(n+1)}&\max_{K}\log\prod_{j\in K}\frac{\|\bbf(z)\|^d\|Q_j\|}{|Q_j(\bbf)(z)|}\\
&\le\dfrac{1}{uH_Y(u)}\left (\max_{\mathcal J\subset\mathcal L}\log\prod_{L\in\mathcal J}\dfrac{\|\bbF(z)\|\cdot \|L\|}{|L(\bbF(z))|}-H_Y(u)\log \|\bbF(z)\|\right )\\
&+d\log \|\bbf(z)\|+\frac{(2n+1)D}{u}\sum_{1\le j\le q}\log\frac{\|\bbf(z)\|^d\|Q_{j}\|}{|Q_{j}(\bbf)(z)|}+O(1),
\end{split}
\end{align}
where $O(1)$ does not depend on $z$. 

By Theorem \ref{2.3}, for any $\epsilon'>0, r>s>0$ large enough, we have
\begin{align}\label{3.7}
\begin{split}
&\| \int_{M\langle r\rangle} \max _{\mathcal{J} \subset \mathcal{L}} \log \prod_{L \in \mathcal{J}} \frac{\|\bbF\| \cdot\|L\|}{\|L(\bbF)\|} \sigma \le\left(H_{Y}(u)+\epsilon'\right) T_{F}(r,s)-N_{ramF}(r,s) \\
&+\left(\frac{H_{Y}(u)\left(H_{Y}(u)-1\right)}{2}+\epsilon'\right)\bigl[m_0(\mathfrak{L};r,s)+{\rm Ric}_{p}(r,s)\\
&+\kappa \log ^+ \mathcal{Y}\left(r^{2}\right)+\kappa \log ^+(r)\bigl]+O(1),
\end{split}
\end{align}
where the maximum is taken over all subsets $\mathcal{J} \subset \mathcal{L}$ so that $\{L: L \in \mathcal{J}\}$ are linearly independent.Taking integral (\ref{3.6}) and combining the result with (\ref{3.7}), we obtain
\begin{align}\label{3.8}
\begin{split}
&\| \int_{M\langle r\rangle} \max _{K \in\mathcal{K}} \sum_{j \in K} \lambda_{Q_{j}}(f) \sigma \le\delta(n+1) d T_{f}(r,s) \\
&+\frac{(2 n+1)D\delta(n+1)}{u} \sum_{j=1}^{q} m_{f}(r,Q_j)+\epsilon' \frac{\delta(n+1)}{H_{Y}(u) u} T_{F}(r,s) \\
&+\frac{\delta(n+1)}{u H_{Y}(u)}\left(\frac{H_{Y}(u)(H_{Y}(u)-1)}{2}+\epsilon'\right)\bigl[m_0(\mathfrak{L};r,s)\\
& +{\rm Ric}_{p}(r,s)+\kappa \log ^+ \mathcal{Y}\left(r^{2}\right)+\kappa \log ^+(r)\bigl]+O(1).
\end{split}
\end{align}
Choose $u $ large enough such that
\begin{align}\label{3.9}
\frac{(2n+1)D\delta(n+1)}{u}dq+\epsilon' \frac{\delta(n+1)}{H_{Y}(u) u}du\le\frac{\epsilon}{2}.
\end{align}
Since $T_{F}(r,s)=d u T_{f}(r,s)$ and $\sum_{j=1}^{q} m_{f}(r, D_{j})\le dq T_{f}(r,s)+O(1)$,
from (\ref{3.8}) and (\ref{3.9}), we have
\begin{align*}
&\| \int_{M\langle r\rangle} \max _{K \in\mathcal K} \sum_{j \in K} \lambda_{D_{j}}(f) \sigma \le(\delta(n+1)+\epsilon) d T_{f}(r,s) \\
&+\frac{\delta(n+1)}{u_0 H_{Y}(u)}\left(\frac{H_{Y}(u)\left(H_{Y}(u)-1\right)}{2}+\epsilon'\right)\bigl[m_0(\mathfrak{L};r,s) \\
& +{\rm Ric}_{p}(r,s)+\kappa \log ^+ \mathcal{Y}\left(r^{2}\right)+\kappa \log ^+(r)\bigl] .
\end{align*}
The theorem is proved.
\end{proof}

We may restate Theorem \ref{3.1} as the following. 
\begin{theorem}\label{3.10}
Let $V$ be a complex projective variety of dimension $n \geq 1$ defined over $\C$. Let $f$ be an algebraically non-degenerate meromorphic mapping from a generalized $p$-Parabolic manifold $(M,\tau,\omega)$, which satisfies the conditions $(\mathcal{A}_{1})$ and $(\mathcal{A}_{2})$, into $V$. Suppose that there exists an ample Cartier divisor $A$ on $V$. Let $D_{1}, \ldots, D_{q}$ be effective Cartier arbitrary divisors in $V$ satisfying that $D_j\sim d_jA\ (1\le j\le q)$, where $d_1,\ldots,d_q$ are positive integers. Then, for any $\delta\ge 1$, $\epsilon>0$ and $r>s>0$, we have
\begin{align*}
\| \int_{M\langle r\rangle}&\max _{K \subset \mathcal{K}} \sum_{j \in K}\frac{1}{d_j}\lambda_{D_{j}}(f) \sigma \le(\delta_V(n+1)+\epsilon) T_{f,A}(r,s)\\
&+c\left(m_0(\mathfrak{L};r,s)+{\rm Ric}_{p}(r,s)+\kappa \log ^+ \mathcal{Y}\left(r^{2}\right)+k \log ^+ r\right),
\end{align*}
where $\mathcal{K}$ is the set of all subsets $K \subset\{1, \ldots, q\}$ so that the distributive constant of the family $\{Y_j;j\in K\}$ with respect to $V$ does not exceed $\delta$ and $c \ge 1$ is a positive constant.
\end{theorem}


\begin{proof}[Proof of Theorem \ref{1.8}]
Denote by $\mathcal{I}_{i}$ the ideal sheaf of $Y_{i}$. Let $\pi_{i}: \tilde V_i\rightarrow V$ be the blowing-up of $V$ along $Y_{i}$ and $E_{i}$ be the exceptional divisor on $\tilde V_i$. Fix the real number $\epsilon>0$. Choose a rational number $\gamma>0$ such that
$$\delta(n+1+\gamma)(1+\gamma)<\epsilon.$$
Then for a small enough positive rational number $\gamma'$ depending on $\gamma, \gamma\pi^* A-\gamma'E_{i}$ is $\mathbb{Q}$-ample on $\tilde V_i$ for all $i \in\{1, \ldots, q\}$. By the definition of Seshadri constant, there exists a rational number $\epsilon_{i}>0$ such that $\epsilon_{i}+\gamma' \geq \epsilon_{Y_{i}}(A)$ and $\pi_{i}^* A-\epsilon_{i} E_{i}$ is $\mathbb{Q}$-nef on $\tilde V_i$ for all $0 \le i \le l$. With such choices, we have $(1+\gamma) \pi_{i}^* A-\left(\epsilon_{i}+\gamma'\right) E_{i}$ is an ample $\mathbb{Q}$-divisor on $\tilde V_i$ for all $i$. Let $N$ be an integer large enough such that $N(1+\gamma) \pi_{i}^* A$ and $N\left[(1+\gamma) \pi_{i}^* A-\left(\epsilon_{i}+\gamma'\right) E_{i}\right]$ are very ample integral divisors on $\tilde V_i$ for all $0 \le i \le q$.

For each subset $K\in\mathcal K$, set $q_K:=\sharp K$ and write $\{Y_j;j\in K\}=\{Y_{0,K},\ldots,Y_{q_K,K}\}$. We also write $\pi_{j,K}=\pi_{i}, \tilde{V}_{i,K}=\tilde V_i$ and $E_{i,K}=E_i$ if $Y_i=Y_{j,K}$. We claim that there are divisors $F_{0,K}, \ldots, F_{q_K,K}$ on $V$ such that:
\begin{itemize}
\item[(a)] $N(1+\epsilon) A \sim F_{i,K}$ and $\pi_{i}^* F_{i,K} \geq N\left(\epsilon_{i}+\gamma'\right) E_{i,K}$ on $\tilde V_{i,K}$ for all $i=0, \ldots, q_K$;
\item[(b)] The family of divisors $\{F_{0,K}, \ldots, F_{q,K}\}$ has the distributive constant with respect to $V$ not exceed $\delta$.
\end{itemize}
We construct these divisors by induction as follows. Assume $F_{0,K}, \ldots, F_{j-1,K}$ have been constructed so that the assertion (a) holds for all $i=1,\ldots,j-1$ and the distributive constant with respect to $V$ of the family $\{F_{0,K}, \ldots, F_{j-1,K}, Y_{j,K}, \ldots, Y_{q_K,K}\}$ does not exceed $\delta$. To find $F_{j,K}$, we let $\tilde F_{i,K}^{(j)}=\pi_{j,K}^* F_{i,K}, i=0, \ldots, j-1$, and $\tilde{Y}_{i,K}^{(j)}=\pi_{j,K}^* Y_{i,K}$ for $i=j+1, \ldots, q$. 

We choose $s_{j,K} \in$ $H^0\left(\tilde V_{j,K}, N(1+\gamma)\pi^*A-N\left(\epsilon_{j}+\gamma'\right)E_{j}\right)$ such that $s_{j,K}$ does not vanish entirely on any irreducible components of any $\left(\bigcap_{i \in I}\tilde F_{i,K}\cap\bigcap_{i \in J}\tilde Y_{i,K}\right)$ where not both $I$ and $J$ are empty. 
Put $\tilde F_{j,K}:={\rm div}\left(s_{i,K}\right)+N\left(\epsilon_{j}+\gamma'\right)\tilde Y_{j,K}$. By Lemma \ref{2.5}, we have, for $N$ big enough,
$$H^0\left(V, \mathcal{O}_{V}(N(1+\gamma) A) \otimes \mathcal{I}_{j}^{N\left(\epsilon_{j}+\gamma'\right)}\right)=H^0\left(\tilde V_{j,K}, \mathcal{O}_{\tilde V_{j}}\left(N\left((1+\gamma) \pi_{j,K}^* A-\left(\epsilon_{j}+\gamma'\right) E_{j,K}\right)\right)\right).$$
Then, there exists a divisor $F_{j,K}$ on $V$ such that $\tilde F_{j,K}=\pi^*F_{j,K}$. Consider an irreducible component with maximal dimension $\Gamma$ of $F_{j,K}\cap\bigcap_{i \in I}F_{i,K}\cap\bigcap_{i \in J} Y_{i,K}$. If $\Gamma\subset Y_{j,K}$ then we have
$$\dim\Gamma\le \dim Y_{j,K}\cap\bigcap_{i \in I}F_{i,K}\cap\bigcap_{i \in J} Y_{i,K}.$$
If $\Gamma\not\subset Y_{j,K}$ then $\pi^*\Gamma\not\subset{\rm div}\left(s_{i,K}\right)$, and hence
$$\dim\Gamma=\dim (\pi^*\Gamma\setminus E_{j,K})\le \dim(\bigcap_{i \in I}\tilde F_{i,K}\cap\bigcap_{i \in J}\tilde Y_{i,K}\setminus E_{j})-1\le \dim(\bigcap_{i \in I}F_{i,K}\cap\bigcap_{i \in J}Y_{i,K}\setminus Y_{j,K})-1,$$
(since $\pi_{j,K}$ is isomorphic outside $E_{j,K}$). Therefore, we have

\begin{align*}
\max\left\{1,\frac{1+\sharp I+\sharp J}{n-\dim\Gamma}\right\}\le\max&\biggl\{1; \frac{1+\sharp I+\sharp J}{n-\dim Y_{j,K}\cap\bigcap_{i \in I}F_{i}\cap\bigcap_{j \in J} Y_{j,K}};\\
&\frac{\sharp I+\sharp J}{n-\dim(\bigcap_{i \in I}F_{i,K}\cap\bigcap_{j \in J}Y_{j,K}\setminus Y_{j,K})}\biggl\}.
\end{align*}
Hence, the distributive constant of $\{F_{0,K}, \ldots, F_{j,K}, Y_{j+1,K}, \ldots, Y_{q_K,K}\}$ with respect to $V$ does not exceed that of $\{F_{0,K}, \ldots, F_{j-1,K}, Y_{j,K}, \ldots, Y_{q_K,K}\}$, in particular not exceed $\delta$. Also, by the construction, we have $\pi_{j,K}^* F_{j,K} \geq N\left(\epsilon_{j}+\gamma'\right) E_{j,K}$ on $\tilde{V}_{j,K}$. So the claim holds by induction.

Since $\pi_{j,K}^* F_{j,K} \geq N\left(\epsilon_{j}+\gamma'\right) E_{j,K}$ on $\tilde V_{i,K}$, by applying Lemma \ref{2.4}, for every point $P \in \tilde V_{j,K} \backslash {\rm Supp} E_{j,K}$, we have
\begin{align*}
N \epsilon_{Y_{j,K}}(A) \lambda_{Y_{j,K}}\left(\pi_{j,K}(P)\right) & \le N\left(\epsilon_{j}+\gamma'\right) \lambda_{Y_{j,K}}\left(\pi_{j,K}(P)\right)=N\left(\epsilon_{j}+\gamma'\right) \lambda_{E_{j,K}}(P) \\
& \le \lambda_{\pi_{j}^* F_{j,K}}(P)=\lambda_{F_{j,K}}\left(\pi_{j,K}(P)\right).
\end{align*}
Then, for every $x\not\in\bigcup_{j,K}f^{-1}(Y_{j,K})$, we have
$$\sum_{j=0}^{q_K}\epsilon_{Y_{j,K}}(A) \lambda_{Y_{j,K}}(f(x))\le \sum_{j=0}^{q_K} \lambda_{F_{j,K}}(f(x)).$$
It implies that
\begin{align}\label{3.11}
\| \int_{M\langle r\rangle}\max _{K \in\mathcal{K}}\sum_{j \in K}\epsilon_{Y_j}(A) \lambda_{Y_j}(f(x))\lambda_{Y_{j}}(f)\sigma\le \| \int_{M\langle r\rangle}\max _{K\in\mathcal{K}}\sum_{j \in K}\lambda_{F_{j,K}}(f)\sigma.
\end{align}
Also, applying Theorem \ref{3.10}, for any $r>s>0$, we have
\begin{align}\label{3.12}
\begin{split}
\| \int_{M\langle r\rangle}&\max _{K \subset \mathcal{K}}\lambda_{F_{j,K}}(f)\sigma \le(\delta(n+1)+\epsilon)T_{f,A}(r,s)\\
&+c\left(m_0(\mathfrak{L};r,s)+{\rm Ric}_{p}(r,s)+\kappa \log ^+ \mathcal{Y}\left(r^{2}\right)+k \log ^+ r\right),
\end{split}
\end{align}
where $c\ge 1$ is a positive constant. Then, the theorem follows from (\ref{3.11}) and (\ref{3.12}). 
\end{proof}

From the proof of Theorem \ref{3.1} and using the technique of \cite{Qmm} on estimate the truncation level, we get the following truncated second main theorem for hypersurfaces.
\begin{corollary}\label{3.13}
Let $V \subset \P^N(\C)$ be a smooth complex projective variety of dimension $n \geq 1$ and $(M,\tau,\omega)$ be an $p$-Parabolic manifold. Let $f: M \rightarrow V$ be an algebraically nondegenerate meromorphic mapping which satisfies the conditions $(\mathcal{A}_{1})$ and $(\mathcal{A}_{2})$. Let $\{Q_{1}, \ldots, Q_{q}\}$ be a family of $q$ hypersurfaces of $\P^N(\C)$ with the distributive constant with respect to $V$ equal to $\delta$. Let $d={\rm lcm}(\deg Q_1,\ldots,\deg Q_q)$. Then, for any $\epsilon>0$ and $r>s>0$, we have
\begin{align*}
\|(q-\delta(n+1)-\epsilon) T_{f}(r,s) \le&\sum_{j=1}^{q}\frac{1}{\deg Q_j}N_{f}^{[m]}(r,s;Q_{j}) \\
&+c\left(m_0(\mathfrak{L};r, s)+{\rm Ric}_{p}(r,s)+\kappa\log^+ \mathcal{Y}\left(r^{2}\right)+\kappa\log^+r\right),
\end{align*}
where $c\le\delta_Ve^n(2n+4)^{n-1}q^nd^{n^2+n-1}\deg(V)^n\epsilon^{-n+1}$ is a positive constant and 

$$m=\left[d^{n^2+2n}\deg (V)^{n+1}e^n(2n+5)^n(n+1)^nq^n\epsilon^{-n}\right].$$ 
\end{corollary}
\noindent
Here, by $[x]$ we denote the biggest integer not exceed the real number $x$.
\begin{proof}[Sketch of the proof]
As usual argument, we may assume that all $Q_j\ (1\le j\le q)$ have the same degree $d$ and $V\cap\bigcap_{j=1}^q Q_j=\emptyset$. 
From (\ref{3.8}), we have
\begin{align}\label{3.14}
\begin{split}
&\|\ \sum_{j=1}^{q}m_{f}(r, Q_{j}) \le\delta(n+1) d T_{f}(r,s) \\
&-\frac{\delta(n+1)}{u H_{Y}(u)} N_{ramF}(r,s)+\frac{(2 n+1)D\delta(n+1)}{u}\sum_{j=1}^{q} m_{f}\left(r,Q_{j}\right) \\
&+\frac{\delta(n+1)}{u H_{Y}(u)}\left(\frac{H_{Y}(u)\left(H_{Y}(u)-1\right)}{2}+\epsilon'\right)\bigl [m_0(\mathfrak{L};r,s)+{\rm Ric}_{p}(r,s)\\
&+\kappa \log ^+ \mathcal{Y}\left(r^{2}\right)+\kappa \log ^+(r)\bigl]+\epsilon' \frac{\delta(n+1)}{H_{Y}(u) u}T_{F}(r,s)+O(1).
\end{split}
\end{align}
Similar as (4.14) in \cite{Qmm}, we have
\begin{align*}
\frac{\delta(n+1)}{duH_{Y}(u)}N_{ramF}(r,s)&\geq \frac{1}{d}\sum_{j=1}^{q}[N_{f}(r,s;Q_{j})-N_{f}^{[n_{u}]}(r,s;Q_{j})] \\
&-\frac{\delta(2 n+1)(n+1)D}{du} \sum_{j=1}^{q}N_{f}(r,s;Q_{j}) .
\end{align*}
From (\ref{3.14}), the above inequality and the first main theorem, we obtain
\begin{align}\label{3.15}
\begin{split}
&\| \quad d(q-\delta(n+1)) T_{f}(r,s) \le \sum_{j=1}^{q} N_{f}^{[n_{u}]}\left(r,s;Q_{j}\right) \\
&+\left(\epsilon' \frac{d\delta(n+1)}{H_{Y}(u)}+\frac{\delta(2 n+1)(n+1)qD}{u}\right)T_{f}(r,s) \\
&+\frac{\delta(n+1)}{uH_{Y}(u)}\left(\frac{H_{Y}(u)\left(H_{Y}(u)-1\right)}{2}+\epsilon'\right)\bigl[m_0(\mathcal{L};r,s)+{\rm Ric}_{p}(r,s)\\
&+\kappa \log ^+ \mathcal{Y}\left(r^{2}\right)+\kappa \log ^+(r)\bigl]+O(1).
\end{split}
\end{align}
We now choose $u$ the smallest integer such that
\begin{align}\label{3.16}
u>d(2n+1)(n+1)qd^n\deg(V)\epsilon^{-1}\ge d(2n+1)(n+1)qD\epsilon^{-1},
\end{align}
and take
$$\epsilon'=\dfrac{H_Y(u)}{d(n+1)}\left (\epsilon-\frac{d(2n+1)(n+1)qD}{u}\right)>0.$$
The number $n_u$ and the constant $c$ are estimated as follows:
\begin{align}\label{3.17}
\begin{split}
n_u&=H_Y(u)-1\le D\binom{n+u}{n}-1\le d^n\deg (V)e^{n}\left(1+\frac{u}{n}\right)^{n}-1\\
&<d^n\deg (V)e^n\left (d(2n+5)qd^n\deg(V)\epsilon^{-1}\right)^n-1\\
&\le \left[\deg (V)^{n+1}e^nd^{n^2+2n}(2n+5)^nq^n\epsilon^{-n}\right]=m;
\end{split}
\end{align}
\begin{align}\label{3.18}
\begin{split}
c&= \frac{\delta(n+1)}{u H_{Y}(u)}\left(\frac{H_{Y}(u)(H_{Y}(u)-1)}{2}+\epsilon'\right)\\
&\le \frac{\delta(n+1)H_{Y}(u)}{2u}=\frac{\delta\binom{n+u}{u}\epsilon}{2d(2n+1)}\\
&\le \delta e^n(2n+4)^{n-1}q^nd^{n^2+n-1}\deg(V)^n\epsilon^{-n+1}.
\end{split}
\end{align} 
Combining (\ref{3.15})-(\ref{3.18}), we obtain the desired inequality of the corollary.
\end{proof}

\section*{Acknowledgements} 
This research is funded by Vietnam National Foundation for Science and Technology Development (NAFOSTED) under grant number 101.02-2021.12.

\noindent
{\bf Disclosure statement.} No potential conflict of interest was reported by the author(s).


\begin{thebibliography}{99}

\bibitem{CT} W. Chen and N. V. Thin, ``A general form of the second main theorem for meromorphic mappings from a $p$-Parabolic manifold to a projective algebraic variety'', \textit{Indian Journal of Pure and Applied Mathematics} \textbf{52} (2021), p. 847--860.

\bibitem{HL} G. Heier, A. Levin, ``A generalized Schmidt subspace theorem for closed subschemes'', \textit{American J. Math.} \textbf{143} (1) (2021), p. 213--226.

\bibitem{EverFer2} J. H. Evertse and R. G. Ferretti, ``Diophantine inequalities on projective varieties'', \textit{Internat. Math. Res. Notices} \textbf{25} (2002), p. 1295--1330.

\bibitem{Ca} H. Cartan, ``Sur les z\'{e}roes des combinaisons lin\'{e}aries de p fonctions holomorphes donn\'{e}es'', \textit{Mathematica} \textbf{7} (1933), p. 80--103.

\bibitem{Han} Q. Han, ``A defect relation for meromorphic maps on generalized p-parabolic manifolds intersecting hypersurfaces in complex projective algebraic varieties'', \textit{Proc. Edinb. Math. Soc.} \textbf{56} (2013), p. 551--574.

\bibitem{Laz17} R.K. Lazarsfeld, ``Positivity in algebraic geometry I'', Volume 48 of  \textit{Ergebnisse der Mathematik und ihrer Grenzgebiete. 3. Folge/ A Series of Modern Surveys in Mathematics} (Springer, 2017).

\bibitem{Qmm} S. D. Quang, ``Quantitative subspace theorem and general form of second main theorem for higher degree polynomials'', \textit{Manuscripta Math.} \textbf{169}(3) (2022), p. 519--547.

\bibitem{Qpcf} S. D. Quang, ``Generalizations of degeneracy second main theorem and Schmidt's subspace theorem'', Pacific J. Math. Volume 318(1) (2022), p. 153--188.

\bibitem{Qjnt} S. D. Quang, ``An effective function field version of Schmidt's subspace theorem for projective varieties, with arbitrary families of homogenous polynomials'', \textit{J. Number Theory} \textbf{241} (2022), p. 563--580.

\bibitem{Ru09} M. Ru, ``Holomorphic curves into algebraic varieties'', \textit{Ann. of Math.} \textbf{169} (2009), p. 255--267. 

\bibitem{Sil87} J. H. Silverman, ``Arithmetic distance functions and height functions in diophantine geometry'', \textit{Mathematische Annalen} \textbf{279}(2) (1987), p. 193--216.

\bibitem{St77} W. Stoll, ``Value distribution on parabolic spaces'', \textit{Lecture Notes in Mathematics} Vol. \textbf{600} (Springer, 1977).

\bibitem{St81} W. Stoll, ``The Ahlfors-Weyl theory of meromorphic maps on parabolic manifolds'', \textit{Lecture Notes in Mathematics}, Vol. \textbf{981}, p. 101--219 (Springer, 1983).

\bibitem{V87} P. Vojta, ``Diophantine approximation and Value Distribution Theory'', \textit{Lecture Notes in Mathematics}, Vol. \textbf{1239} (Springer, Berlin, 1987).

\bibitem{WCC} L. Wang, T. Cao and H. Cao, ``A generalized second main theorem for closed subschemes'', Ann. Pol. Math. \textbf{129}(3) (2022), 275--290.

\bibitem{WW} P. M. Wong and P. P. W. Wong, ``The second main theorem on generalized parabolic manifolds, in some topics on value distribution and differentiability in complex and $p$-adic analysis'', \textit{Mathematical Monographs series}, Vol. \textbf{11}, p. 3-41 (Science Press, Beijing, 2008).

\bibitem{Y} K.Yamanoi, ``Algebro geometric version of Nevanlinna's lemma on logarithmic derivative and applications'', \textit{Nagoya Math. J.} \textbf{173} (2004), p. 23--63.
\end{thebibliography}
\end{document}